\documentclass[reqno,11pt,a4paper]{amsart}
\usepackage{graphicx,color,amssymb,latexsym,amsfonts,amsmath}
\usepackage{mathrsfs}
\usepackage{graphics}
\usepackage{lscape}
\usepackage{dsfont}
\usepackage[normalem]{ulem} 
\usepackage{verbatim}

\newtheorem{lemma}{\bf Lemma}[section]

\newtheorem{theorem}{\bf Theorem}[section]

\numberwithin{equation}{section} \theoremstyle{plain}
\theoremstyle{definition}

\DeclareMathOperator*{\argmax}{argmax}

\begin{document}

\title[] {Large Deviations  Application to Billingsley's Example}
\author{R. Liptser}
\address{Department of Electrical Engineering Systems,
Tel Aviv University, 69978 Tel Aviv, Israel}
\email{liptser@eng.tau.ac.il; rliptser@gmail.com}

\keywords{Empirical distribution; LDP; Stopping time. }
\subjclass{60F10, 60J27}
%
\maketitle
\begin{abstract}
We consider a classical model related to an empirical distribution
function $ F_n(t)=\frac{1}{n}\sum_{k=1}^nI_{\{\xi_k\le t\}}$ of
$(\xi_k)_{i\ge 1}$ -- i.i.d. sequence of random variables, supported
on the interval $[0,1]$, with continuous distribution function
$F(t)=\mathsf{P}(\xi_1\le t)$. Applying ``Stopping Time
Techniques'', we give a proof of Kolmogorov's exponential bound
$$
\mathsf{P}\big(\sup_{t\in[0,1]}|F_n(t)-F(t)|\ge \varepsilon\big)\le
\text{const.}e^{-n\delta_\varepsilon}
$$
conjectured by Kolmogorov in
1943. Using this bound we establish a best possible logarithmic
asymptotic of
$$
\mathsf{P}\big(\sup_{t\in[0,1]}n^\alpha|F_n(t)-F(t)|\ge
\varepsilon\big)
$$
with rate $ \frac{1}{n^{1-2\alpha}}
$
slower than
$\frac{1}{n}$ for any $\alpha\in\big(0,\frac{1}{2}\big)$.
\end{abstract}

\section{\bf Introduction}
\label{sec-1}

Let $(\xi_k)_{i\ge 1}$ be the i.i.d. sequence of random variables with values in
the interval $[0,1]$ having a continuous distribution function $F(t)=\mathsf{P}(\xi_1\le t)$.
Consider an empirical distribution $
F_n(t)=\frac{1}{n}\sum\limits_{k=1}^nI_{\{\xi_k\le t\}}. $ A strong
law of  large numbers for sums of i.i.d. random variables guaranties
that for any $t\in[0,1]$, $ F_n(t)\xrightarrow[n\to \infty]{\rm
a.s.}F(t) $ and the Glivenko-Cantelli theorem also guarantees a uniform
convergence $
\sup_{t\in[0,1]}|F_n(t)-F(t)|\xrightarrow[n\to\infty]{\rm a.s.}0. $

For any fixed $t$, the rate of convergence, in $n\to\infty$, of
$|F_n(t)-F(t)|$ is also well known from the Central Limit Theorem (CLT):
$\{\frac{1}{\sqrt{n}}[F_n(t)-F(t)]\}_{n\to \infty}$ converges in law
to a zero mean Gaussian random variable with the variance $
F(t)[1-F(t)] $.

From Theorem 16.4 of Billingsley (1968), it is known that the family
$\big\{\big(\frac{1}{\sqrt{n}}[F_n(t)-F(t)]\big)_{t\in
[0,1]}\big\}_{n\to \infty}$ converges in law (in Skorokhod's and
uniform metrics) to a zero mean Markov-Gaussian process
$X=(X_t)_{t\in[0,1]}$ with a correlation function
\begin{equation}\label{KTS}
K(t,s)=F(s\wedge t)[1-F(s\vee t)].
\end{equation}
For $F(t)\equiv t$, the limit random process $X$ is known as
``Brownian Bridge'' defined as the unique solution of It\^o's
equation $ X_t=-\int_0^t\frac{X_s}{1-s}ds+B_t $ relative to Brownian
motion $B_t$. In the general case, $F(t)\not\equiv t$, the random
process $X$ can be also defined as the unique solution of It\^o's
equation
\begin{equation}\label{eqXt}
X_t=-\int_0^t\frac{X_s}{1-F(s)}dF(s)+\mathbf{M}_t
\end{equation}
with Brownian motion $B_t$ replaced by a Gaussian martingale
$\mathbf{M}_t$, $ \mathsf{E} \mathbf{M}^2_t\equiv F(t) $ (see
Section \ref{sec-2.1}).

Once, Prof. A.N. Shiryaev has mentioned to participants of the Probability Seminar at the Steklov Mathematical Institute that in 1943 Kolmogorov conjectured the following rate of convergence in the uniform metric,
\begin{equation}\label{eq:KOL}
\mathsf{P}\bigg(\sup_{t\in[0,1]}|F_n(t)-F(t)|\ge
\varepsilon\bigg)\le \text{const.}e^{-n\delta_\varepsilon} ,
\end{equation}
a proof of which has never been published.

In this paper, we give a version of Kolmogorov's
exponential bound  with
\begin{equation*}
\delta_\varepsilon=\frac{\varepsilon}{8}\Big\{\log\Big(1+\frac{\varepsilon^2}{32}\Big)-1\Big\}
+\frac{4}{\varepsilon}\log\Big(1+\frac{\varepsilon^2}{32}\Big).
\end{equation*}
It should be noted that neither Sanov's theorem (1961), \cite{San} (see also Dembo Zeitouni,
\cite{DZ}) nor Wu's result (1994), \cite{Wu}, are not relevant
tools for obtaining the Kolmogorov  bound \eqref{eq:KOL}, since the
Levy-Prohorov  metric is involved in Sanov (1961) and Wu (1994). A
crucial role in proving of \eqref{eq:KOL} plays ``Stopping Time
Techniques''.

Unfortunately, we
could not claim that \eqref{eq:KOL} is best
possible bound even in a logarithmic scale. However, the Kolmogorov bound
helps us to establish the following logarithmic asymptotics: for any $\alpha\in\big(0,\frac{1}{2}\big)$ and any $T$ in a small vicinity of
$\{1\}$,
\begin{gather}\label{eq:1.5}
\lim_{n\to\infty}\frac{1}{n^{1-2\alpha}}\log\mathsf{P}\bigg(\sup_{t\in[0,T]}n^\alpha|F_n(t)-F(t)|
\ge \varepsilon\bigg) =-2\varepsilon^2.
\end{gather}
We build the proof of \eqref{eq:1.5} based on Kolmogorov's bound and on a non-standard Large Deviations technique. A key for \eqref{eq:1.5} consists in choosing  the rate
$\frac{1}{n^{1-\alpha}}$ slower than $\frac{1}{n}$.

The paper is organized as follows. Section \ref{sec-2} contains
auxiliary results from Stochastic Calculus useful for the asymptotic analysis of  the random
process $\big(F_n(t)-F(t)\big)_{t\in[0,T]}$ as $n\to\infty$. Proofs of
\eqref{eq:KOL} and \eqref{eq:1.5} are given in Sections \ref{sec-4a}
and \ref{sec-5} respectively. The Large Deviations Principle  result is formulated and
proved in Section \ref{sec-3} (Appendix).

\section{\bf Stochastic calculus applicability to Billingsley's theorem }
\label{sec-2}

\subsection{$\pmb{X_t}$ as a Solution of (\ref{eqXt})}
\label{sec-2.1}

The limit random processes $X=(X_t)_{t\in[0,1]}$ is zero mean
Gaussian with the correlation function defined in \eqref{KTS}. By
Theorem 8.1 of Doob (1953), the gaussianity of $X$ jointly with an obvious
property of the correlation function,
$$
K(t,s)=\frac{K(t,u)K(u,s)}{K(u,u)},
$$
 enable us to claim
that $X$ is Markov process with respect to a minimal filtration
$(\mathscr{F}^X_t)_{t\in[0,1]}$ generated by $X$. Then for
$s<u<t$,
$$
\mathsf{E}\Big(\frac{X_t}{1-F(t)}\Big|\mathscr{F}^X_u\Big)=\mathsf{E}\Big(\frac{X_t}{1-F(t)}\Big|X_u\Big)
=\frac{1}{1-F(t)}\frac{K(t,u)}{K(u,u)}X_u= \frac{X_u}{1-F(u)}.
$$
In other words, the Gaussian random process $N_t=\frac{X_t}{1-F(t)}$
is the square integrable martingale, i.e., a process with orthogonal
increments (so, with independent increments too). Hence, its
predictable variation process $\langle N\rangle_t$ coincides with $
\mathsf{E} N^2_t=\frac{K(t,t)}{[1-F(t)]^2}=\frac{F(t)}{1-F(t)}. $

Therefore, the process $ \mathbf{M}_t=\int_0^t[1-F(s)]dN_s $ is the
Gaussian martingale with
$$
\langle M\rangle_t=\int_0^t[1-F(s)]^2d\langle
N\rangle_s=\int_0^t[1-F(s)]^2d\Big(\frac{F(s)}{1-F(s)}\Big) =F(t).
$$
Finally, the It\^o equation \eqref{eqXt} is derived by applying the It\^o formula to
$X_t=[1-F(t)]N_t$.

\subsection{Counting Process $\pmb{\sum\limits_{k=1}^nI_{\{\xi_k\le t\}}}$}
\label{sec-2.2}

Without loss of generality we shall assume that all $\xi_k$'s are
defined on a probability space $ (\varOmega,\mathcal{F},\mathsf{P}).
$ Denote
\begin{itemize}
  \item $\mathscr{F}^k=(\mathscr{F}^k_t)_{0\le t\le 1}$ the
filtration  generated by $I_{\{\xi_k<t\}}$,
\item $\mathscr{F}_t=\bigvee_{k\ge 1}\mathscr{F}^k_t$,
  \item $ \mathcal{F}_=\bigvee_{t\in[0,1]}\mathscr{F}_t$
\end{itemize}
and assume also that the general conditions for these filtrations are
fulfilled.

The random process $I_{\{\xi_k\le t\}}$ has piece-wise constant and
right continuous  paths  with only one jump of the unit size. Thus,
$(I_{\{\xi_k\le t\}}, \mathscr{F}^k_t)_{t\in[0,1]}$ is a counting
process with continuous (!) compensator $(A^k_t)_{t\in [0,1]}$,
\begin{equation*}
A^k_t=\int_0^{t\wedge\xi_k}\frac{dF(s)}{1-F(s)}=\int_0^t\frac{1-I_{\{\xi_k\le
s\}}}{1-F(s)}dF(s)
\end{equation*}
(see, e.g., formula (18.23), Section 18.2 in
\cite{LSII}). Set $M^k_t=I_{\{\xi_k\le
t\}}-A^k_t$. It is well known (see, e.g., Ch. 18 in
\cite{LSII}) that $(M^k_t,
\mathscr{F}^k_t)_{t\in[0,1]}$ is a square integrable martingale with
paths from the Skorokhod space $\mathbb{D}_{[0,1]}$ and its
predictable quadratic variation process $\langle
M^k\rangle_t\equiv A^k_t$. The joint independence of $(\xi_k)_{k\ge 1}$
implies that $ \{(I_{\{\xi_k\le t\}}, \mathscr{F}_t)_{t\in [0,1]}\}_{k\ge
1} $ are counting processes with disjoint jumps. Set $
\mathbf{I}^n_t=\sum\limits_{k=1}^nI_{\{\xi_k\le t\}} $. Then, $
(\mathbf{I}^n_t, \mathscr{F}_t)_{t\in [0,1]} $ is a counting process
with the corresponding compensator,
\begin{equation}\label{eq:AM}
\mathbf{A}^n_t=\sum_{k=1}^nA^k_t=n\int_0^t\frac{1-F_n(s)}{1-F(s)}dF(s)
\end{equation}
or, equivalently, $ \big(\mathbf{I}^n_t-\mathbf{A}^n_t,
\mathscr{F}_t\big)_{t\in[0,1]} $ is the square integrable martingale
with the predictable variation process $ \mathbf{A}^n_t. $ Two other
martingales are related to $ \big(\mathbf{I}^n_t-\mathbf{A}^n_t,
\mathscr{F}_t\big)_{t\in[0,1]}: $ $
(\mathbf{M}^n_t,\mathscr{F}_t)_{t\in[0,1]} $ and $
(\mathbf{M}^{n,\alpha}_t,\mathscr{F}_t)_{t\in[0,1]}, $ where
\begin{equation*}
\mathbf{M}^n_t=\frac{1}{\sqrt{n}}\big(\mathbf{I}^n_t-\mathbf{A}^n_t\big)\quad
\text{and}\quad
\mathbf{M}^{n,\alpha}_t=\frac{1}{n^{\frac{1}{2}-\alpha}}\mathbf{M}^{n}_t,
\ \alpha\in \Big[0,\frac{1}{2}\Big),
\end{equation*}
with predictable variation processes respectively:
$$
\langle \mathbf{M}^n\rangle=\frac{1}{n}\mathbf{A}^n_t\quad
\text{and}\quad \langle
\mathbf{M}^{n,\alpha}\rangle_t=\frac{1}{n^{2(1-\alpha)}}\mathbf{A}^n_t.
$$

\subsection{Functional Central Limit Theorem for $\pmb{\mathbf{M}^n_t}$}

\begin{theorem}\label{lem-2.-1}
The family of martingales
$\big\{\big(M^n_t\big)_{t\in[0,1]}\big)\big\}_{n\to\infty}$
converges in law {\rm (}in Skorokhod's and uniform metrics{\rm )}
to a Gaussian martingale  $\mathbf{M}_t$ with $\langle
\mathbf{M}\rangle_t=F(t)$.
\end{theorem}

\begin{proof}
In view of the function $F(t)$ is continuous, the Gaussian martingale $\mathbf{M}_t$ is
continuous too. Then, by Theorem 2, Ch. 7, \S 1  of Liptser-Shiryaev
(1989), \cite{LSMar}, the desired statement holds true provided that $
\big\langle \mathbf{M}^n\big\rangle_t \xrightarrow[n\to\infty]{\rm
prob.}F(t), \ \forall \ t\in[0,1]. $ The latter holds since
$$
\big\langle \mathbf{M}^n\big\rangle_t=\frac{1}{n}\mathbf{A}^n_t
=\displaystyle{\frac{1}{n}\sum_{k=1}^n\int_0^t}\frac{1-I_{\{\xi_k\le
s\}}}{1-F(s)}dF(s)
$$
and, in the case under consideration, the strong law of large numbers for sums of i.i.d. random
variables implies
\begin{gather*}
\lim_{n\to\infty}\Big\langle \mathbf{M}^n\Big\rangle_t=
\mathsf{E}\int_0^t\frac{1-I_{\{\xi_1\le s\}}}{1-F(s)}dF(s) =F(t) \
\text{a.s.} \ \forall \ t\in[0,1].
\end{gather*}
\end{proof}

\subsection{Semimartingale Decomposition of Centered Empirical Distribution }

Set
\begin{equation}\label{eq:2.2a}
X^{n,\alpha}_t=n^\alpha\big[F_n(t)-F(t)\big], \
\alpha\in\Big[0,\frac{1}{2}\Big].
\end{equation}

\begin{lemma}\label{lem-2.1}
For \ $t\in[0,1)$,
\begin{itemize}
  \item[{\bf (i)}]
$
X^{n,\alpha}_t=-\displaystyle{\int_0^t}\frac{X^{n,\alpha}_s}{1-F(s)}dF(s)+
\frac{1}{n^{\frac{1}{2}-\alpha}}\mathbf{M}^{n}_t; $
\item[{\bf (ii)}]
  $
X^{n,\alpha}_t=\frac{1}{n^{\frac{1}{2}-\alpha}}[1-F(t)]\displaystyle{\int_0^t}\frac{d\mathbf{M}^{n}_s}{1-F(s)};
$
\item[{\bf (iii)}]
$
X^{n,\alpha}_t=\frac{1}{n^{\frac{1}{2}-\alpha}}\Big\{\mathbf{M}^{n}_t-[1-F(t)]\displaystyle{\int_0^t}
\frac{\mathbf{M}^{n}_s}{[1-F(s)]^2}dF(s)\Big\}; $
\item[{\bf (iv)}] $X^{n,\alpha}_t=\mathsf{\Psi}\Big(\frac{1}{n^{\frac{1}{2}-\alpha}}\mathbf{M}^{n}_{[0,t]}\Big)$,
where for any function $(x_t)_{t\in[0,1]}$ from the Skorokhod space
$\mathbb{D}_{[0,1]}$,
\begin{equation*}
\mathsf{\Psi}\big(x_{[0,t]}\big)=x_t-[1-F(t)]\int_0^t\frac{x_s}{[1-F(s)]^2}dF(s)
\end{equation*}
is continuous function in the uniform metric on $[0,1]$.
\end{itemize}
\end{lemma}

\begin{proof}
{\bf (i)} From \eqref{eq:2.2a} and the definition of
$\mathbf{A}^n_t$ and $\mathbf{M}^n_t$, it follows that $
X^{n,\alpha}_t=\frac{1}{n^{1-\alpha}}\sum\limits_{k=1}^n[A^k_t-F(t)]+\frac{1}{n^{\frac{1}{2}-\alpha}}
\mathbf{M}^{n}_t. $ Consequently,
$$
\begin{aligned}
&
\frac{1}{n^{1-\alpha}}\sum_{k=1}^n[A^k_t-F(t)]
\\
&
=\frac{1}{n^{1-\alpha}}\sum_{k=1}^n
\bigg[\int_0^{t\wedge\xi_k}\frac{dF(s)}{1-F(s)}-F(t) \bigg]
\\
&=\frac{1}{n^{1-\alpha}}\sum_{k=1}^n\bigg[\int_0^t\frac{1-I_{\{\xi_k\le
s\}}}{1-F(s)}dF(s) -F(t)\bigg]
\\
&=-\int_0^t\frac{n^\alpha[F_n(s)-F(s)]}{1-F(s)}dF(s)
=-\int_0^t\frac{X^{n,\alpha}_s}{1-F(s)}dF(s).
\end{aligned}
$$

\smallskip
{\bf (ii)} This formula describes the unique solution of It\^o's
equation from {\bf (ii)}

\smallskip
{\bf (iii)} The It\^o formula $
\frac{\mathbf{M}^{n}_t}{1-F(t)}=\int_0^t\frac{d\mathbf{M}^{n}_s}{1-F(s)}
+\int_0^t\frac{\mathbf{M}^{n}_s}{[1-F(s)]^2}dF(s) $ and {\bf (ii)}
provide
\begin{align*}
\frac{1}{n^{\frac{1}{2}-\alpha}}\mathbf{M}^{n}_t&=\frac{1}{n^{\frac{1}{2}-\alpha}}[1-F(t)]
\bigg\{\frac{\mathbf{M}^{n}_t}{1-F(t)}\bigg\}
\\
&=\frac{1}{n^{\frac{1}{2}-\alpha}}[1-F(t)]\bigg\{\int_0^t\frac{d\mathbf{M}^{n}_s}{1-F(s)}
+\int_0^t\frac{\mathbf{M}^{n}_s}{[1-F(s)]^2}dF(s)\bigg\}
\\
&=X^{n,\alpha}_t+\frac{1}{n^{\frac{1}{2}-\alpha}}[1-F(t)]\int_0^t\frac{\mathbf{M}^{n}_s}{[1-F(s)]^2}dF(s).
\end{align*}

\smallskip
{\bf (iv)} $\mathsf{\Psi}(x_{[0,t]})$ is nothing but {\bf (iii)}
with $x_t$ replaced by
$\frac{1}{n^{\frac{1}{2}-\alpha}}\mathbf{M}^{n}_t$. A desired
continuity of $\mathsf{\Psi}$ follows from
$$
\sup_{t\in[0,1]}|x'_t-x''_t|\le \varepsilon \ \Rightarrow \
\sup_{t\in[0,1]}\big|\mathsf{\Psi}\big(x'_{[0,t]}\big)-\mathsf{\Psi}\big(x''_{[0,t]}\big)\big|\le
2\varepsilon.
$$
\end{proof}

\subsubsection{\bf An Alternative Proof
of Billingsley's Theorem}

For $\alpha=\frac{1}{2}$, write
$
X^{n,\frac{1}{2}}_t=\sqrt{n}\big[F_n(t)-F(t)\big].
$
\begin{lemma}\label{lem-2.0}
The family
$\big\{\big(X^{n,\frac{1}{2}}_t)_{t\in[0,1]}\big)\big\}_{n\to\infty}$
converges in law {\rm (}in Skorokhod's and uniform metrics{\rm )}
to the continuous Gaussian process $(X_t)_{t\in[0,1]}$ defined in
\eqref{eqXt}.
\end{lemma}
\begin{proof}
By Lemma \ref{lem-2.1}{\bf (iv)}, $
X^{n,\frac{1}{2}}_t=\mathsf{\Psi}\big(\frac{1}{\sqrt{n}}\mathbf{M}^n_{[0,t]}\big)
$ and by Theorem \ref{lem-2.-1},
$$
\big(X^{n,\frac{1}{2}}_t\big)_{t\in[0,1]}\xrightarrow[n\to\infty]{\rm
law} \mathsf{\Psi}\big(\mathbf{M}_{[0,t]}\big)_{t\in[0,1]}.
$$
Now,
by applying the It\^o formula to
$X_t:=\mathsf{\Psi}\big(\mathbf{M}_{[0,t]}\big)$, we make sure that
$X_t$ solves \eqref{eqXt}.
\end{proof}

\section{\bf The Kolmogorov bound}
\label{sec-4a}

In this section, we show that
\begin{gather}\label{eq:KOL+}
\mathsf{P}\bigg(\sup_{t\in[0,1]}|F_n(t)-F(t)|\ge
\varepsilon\bigg)
\nonumber\\
\le
2\exp\Bigg(-n\Big[\frac{\varepsilon}{8}\Big\{\log\Big(1+\frac{\varepsilon^2}{32}\Big)-1\Big\}
+\frac{4}{\varepsilon}\log\Big(1+\frac{\varepsilon^2}{32}\Big)\Big]\bigg).
\end{gather}

\smallskip
\noindent Since $F_n(t)$ and $F(t)$ are increasing functions and $F(t)$ is continuous, the
following upper bound with a free parameter $T\in(0,1)$ holds:
\begin{align*}
\sup_{t\in[0,1]}|F_n(t)-F(t)|&\le
\sup_{t\in[0,T]}|F_n(t)-F(t)|+\sup_{t\in(T,1]}|F_n(t)-F(t)|
\\
& \le \sup_{t\in[0,T]}|F_n(t)-F(t)|+|1-F(T)|+|1-F_n(T)|
\\
& \le \sup_{t\in[0,T]}|F_n(t)-F(t)|+2[1-F(T)]+|F(T)-F_n(T)|
\\
&\le 2\Big\{\sup_{t\in[0,T]}|F_n(t)-F(t)|+[1-F(T)]\Big\}.
\end{align*}
A choice of $T$ with $1-F(T)=\frac{\varepsilon}{4}$ guarantees a
useful upper bound
\begin{gather*}
\mathsf{P}\bigg(\sup_{t\in[0,1]}|F_n(t)-F(t)|\ge \varepsilon\bigg)
\le
\mathsf{P}\bigg(\sup_{t\in[0,T]}|F_n(t)-F(t)|\ge\frac{\varepsilon}{4}\bigg).
\end{gather*}
By Lemma \ref{lem-2.1} ({\bf (iv)}) with $\alpha=0$, we find that
\begin{gather*}
\sup_{t\in[0,T]}|F_n(t)-F(t)|
=\sup_{t\in[0,T]}|X^{n,0}_t|
\\
\le\frac{1}{\sqrt{n}}\sup_{t\in[0,T]}|\mathbf{M}^n_t|\Big(1+\sup_{t\in[0,1]}[1-F(t)]
\int_0^t\frac{dF(s)}{[1-F(s)]^2}\Big)
\\
\le\frac{2}{\sqrt{n}}\sup_{t\in[0,T]}|\mathbf{M}^n_t|
\end{gather*}
and the following upper bound:
$$
\mathsf{P}\big(\sup_{t\in[0,1]}|F_n(t)-F(t)|\ge \varepsilon\big)
\le
\mathsf{P}\big(\sup_{t\in[0,T]}\frac{1}{\sqrt{n}}|\mathbf{M}^n_t|\ge\frac{\varepsilon}{8}\big).
$$

Now, we shall combine ``exponential martingale'' and ``stopping
time'' techniques. With $\lambda>0$, let us introduce the
exponential martingale
\begin{equation}\label{ZZZ}
\mathfrak{z}_t=
\exp\bigg(\frac{\lambda}{\sqrt{n}}\mathbf{M}^n_t-\Big[e^{\frac{\lambda}{n}}-\frac{\lambda}{n}-1\Big]
\mathbf{A}^n_t\bigg)
\end{equation}
relative to the filtration $(\mathscr{F}_t)_{t\in[0,1]}$. It is well
known that any exponential martingale is a supermartingale too, that
is, $(\mathfrak{z}_t,\mathscr{F}_t)_{t\in[0,1]}$ is the nonnegative
supermartingale with $
\mathsf{E}\mathfrak{z}_\tau\le\mathsf{E}\mathfrak{z}_0=1 $ for any
stopping time $\tau$ w.r.t. the filtration
$(\mathscr{F}_t)_{[t\in[0,1]]}$.

We choose two stopping times,
$$
\tau^n_\pm=\inf\Big\{t\le T:\pm
\frac{1}{\sqrt{n}}\mathbf{M}^{n}_t\ge\frac{\varepsilon}{8} \Big\},
\ \ \inf(\varnothing)=\infty,
$$
and use them for obtaining the following bound:
$$
\mathsf{P}\Big(\sup_{t\in[0,T]}\frac{1}{\sqrt{n}}|\mathbf{M}^{n}_t|>
\frac{\varepsilon}{8}
\Big) \le 2\max\Big[\mathsf{P}\Big(\tau^n_+<\infty\Big),
\mathsf{P}\Big(\tau^n_-<\infty\Big)\Big].
$$
In order to find an upper bound of $\mathsf{P}(\tau_+<\infty)$,
write
\begin{align*}
1&\ge \mathsf{E}\mathfrak{z}_{\tau_+}\ge
\mathsf{E}I_{\{\tau_+<\infty\}} \mathfrak{z}_{\tau_+}=
\mathsf{E}I_{\{\tau_+<\infty\}}
\exp\Big(\lambda\frac{1}{\sqrt{n}}\mathbf{M}^{n}_{\tau_+}-\Big[e^{\frac{\lambda}{n}}-1
-\frac{\lambda}{n}\Big]\mathbf{A}^n_{\tau_+}\Big)
\\
&\ge\mathsf{P}(\tau_+<\infty) \exp\Big(\lambda\frac{\varepsilon}{8}
-\Big[e^{\frac{\lambda}{n}}-1-\frac{\lambda}{n}\Big]\mathbf{A}^n_T\Big).
\end{align*}
By \eqref{eq:AM},
$
\mathbf{A}^n_T\le \frac{n}{1-F(T)}=\frac{4n}{\varepsilon},
$
so that
$$
1\ge\mathsf{P}(\tau^n_+<\infty)\exp\Big(\lambda\frac{\varepsilon}{8}
-\Big[e^{\frac{\lambda}{n}}-1-\frac{\lambda}{n}\Big]\frac{4n}{\varepsilon}\Big)
$$
or, equivalently,
$ \mathsf{P}(\tau^n_+<\infty)\le
\exp\big(-\big\{\lambda\frac{\varepsilon}{8}
-\big[e^{\frac{\lambda}{n}}-1-\frac{\lambda}{n}\big]\frac{4n}{\varepsilon}\big\}\big).
$
Since $\lambda$ is an arbitrary positive parameter, we can set $\lambda$ as
$
\lambda^*=\argmax_{\mu>0}
\big\{\mu\frac{\varepsilon}{8}
-\big[e^{\frac{\mu}{n}}-1-\frac{\mu}{n}\big]\frac{4n}{\varepsilon}\big\}
=n\log\big(1+\frac{\varepsilon^2}{32}\big),
$
in order to obtain
\begin{gather*}
\mathsf{P}(\tau^n_+<\infty)
\\
\le
\exp\Big(-\Big\{\lambda^*\frac{\varepsilon}{8}
-\Big[e^{\frac{\lambda^*}{n}}-1-\frac{\lambda^*}{n}\Big]\frac{4n}{\varepsilon}\Big\}\Big)
\\
=\exp\Big(-n\Big[\frac{\varepsilon}{8}\Big\{\log
\Big(1+\frac{\varepsilon^2}{32}\Big)-1\Big\}
+\frac{4}{\varepsilon}\log\Big(1+\frac{\varepsilon^2}{32}\Big)\Big]\Big).
\end{gather*}

\medskip
The proof of the upper bound
$
\mathsf{P}(\tau^n_-<\infty)
\le \exp\big(-n\big[\frac{\varepsilon}{8}\big\{\log
\big(1+\frac{\varepsilon^2}{32}\big)-1\big\}
+\frac{4}{\varepsilon}\log\big(1+\frac{\varepsilon^2}{32}\big)\big]\big)
$
is similar.

\medskip
\noindent
Therefore,
\eqref{eq:KOL+} holds. \qed

\section{\bf The proof of (\ref{eq:1.5})}
\label{sec-5}

Recall that
$
X^{n,\alpha}_t=n^\alpha\big[F_n(t)-F(t)\big]
$
(see \eqref{eq:2.2a}).

\begin{theorem}\label{theo-5.1}
\underline{}For any $\alpha\in\big(0,\frac{1}{2}\big)$ and any $T$ in a small vicinity of $\{1\}$,
\begin{gather*}
\lim_{n\to\infty}\frac{1}{n^{1-2\alpha}}
\log\mathsf{P}\bigg(\sup_{t\in[0,T]}n^\alpha|X^{n,\alpha}_t|\ge
\varepsilon\bigg) =-2\varepsilon^2.
\end{gather*}
\end{theorem}
\begin{proof}
By Theorem \ref{theo-4.2x} (Appendix) the family
$
\{(n^\alpha X^{n,\alpha}_t))_{t\in[0,T]}\}_{n\to \infty}
$
obeys the large deviations principle in the Skorokhod space $\mathbb{D}_{[0,1]}$ relative Skorokhod's
and uniform metrics with
the rate
$\frac{1}{n^{1-2\alpha}}$ and the rate function
\begin{equation*}
J_T(u)=
  \frac{1}{2}\begin{cases}
    \displaystyle{\int_0^T}\Big(\dot{u}_t+\frac{u_t}{1-F(t)}\Big)^2dF(t) , &{\substack{u_0=0\\du_t=
                                                                      \dot{u}_tdF(t)\\
                                                                      \int_0^T(\dot{u}_t+\frac{u_t}{1-F(t)}
                                                                      )^2dF(t)<\infty}}
 \\
    \infty , & \text{otherwise}.
  \end{cases}
\end{equation*}
Since paths of $(X^{n,\alpha})_{t\in[0,T]}$ with property
$
\big\{\sup_{t\in[0,T]}n^\alpha|X^{n,\alpha}_t|\ge
\varepsilon\big\}
$
form  a closed set
$$
\mathsf{C}=\left\{u\in\mathbb{D}_{[0,T]}: \substack{u_0=0
\\ \\
\theta(u)=\inf\{t\le T:|u_t|\ge \varepsilon\}\le T \\ \\ u_t\equiv
0, \ t>\theta(u) } \right\},
$$
in accordance with the large deviations theory,
$$
\varlimsup\limits_{n\to\infty}\frac{1}{n^{1-2\alpha}}\log
\mathsf{P}\big(\sup_{t\in[0,T]}n^\alpha|X^{n,\alpha}_t|
\ge \varepsilon\big) \le -\inf\limits_{u\in \mathsf{C}}J_T(u).
$$
A minimization procedure of $J_T(u)$ in $u\in\mathsf{C}$ automatically excludes from
consideration all functions $(u_t)_{t\in[0,T]}$ with $J_T(u)=\infty$. Consequently,
\begin{multline*}
\varlimsup_{n\to\infty}\frac{1}{n^{1-2\alpha}}\log\mathsf{P}\Big(\sup_{t\in[0,T]}n^\alpha|X^{n,\alpha}_t|
\ge
\varepsilon\Big) \le
\\
-\frac{1}{2}\inf_{u\in \mathsf{C}}\int_0^{T\wedge
\theta(u)}\Big[\dot{u}_s+\frac{u_s}{1-F(s)} \Big]^2dF(s).
\end{multline*}
Denote
$
 w_t=\dot{u}_t+\frac{u_t}{1-F(t)}.
$
Then
$
J_{\theta(u)}(u)=\frac{1}{2}\int_0^{\theta(u)}w^2_tdF(t),
$
and
$$
u_t=-\int_0^t\frac{u_s}{1-F(s)}dF(s)+\int_0^tw_sdF(s), \ t\le \theta(u).
$$
This integral equation obeys the unique solution
$$
u_{t\wedge\theta(u)}=[1-F(t\wedge\theta(u))]\int_0^{t\wedge
\theta(u)}\frac{w_s}{1-F(s)}dF(s).
$$
The assumption $\theta(u)\le T$ implies
$u^2_{\theta(u)}=\varepsilon^2$. Hence
\begin{equation}\label{eq:JJ}
\varepsilon^2=[1-F(\theta(u))]^2\Big(\int_0^{\theta(u)}
\frac{w_s}{1-F(s)}dF(s)\Big)^2.
\end{equation}
Now, we apply the Cauchy-Schwarz's inequality,
\begin{gather}\label{eq:KSW}
\Big(\int_0^{\theta(u)} \frac{w_s}{1-F(s)}dF(s)\Big)^2\le
\int_0^{\theta(u)}\frac{dF(s)}{[1-F(s)]^2}
\int_0^{\theta(u)}w^2_sdF(s)
\\
=\frac{F(\theta(u))}{1-F(\theta(u)} 2J_{\theta(u)}(u),
\nonumber
\end{gather}
transforming \eqref{eq:JJ} into the lower bound:
$
J_{\theta(u)}(u)\ge \frac{\varepsilon^2}{2F(\theta(u))[1-F(\theta(u))]}.
$
Assume for a moment that there exists $u^*_t$ such that $F(\theta(u^*))=\frac{1}{2}$. Then the following
lower bound
$
J_{\theta(u^*)}(u^*)\ge 2\varepsilon^2
$
is valid. This lower bound is attainable,
$
J_{\theta(u^*)}(u^*)= 2\varepsilon^2,
$
provided that the Cauchy-Schwarz's inequality in \eqref{eq:KSW} becomes the equality.
The latter holds true if $w^*_s$, related to $u^*_t (\dot{u}^*_t)$, is in a proportion to $\frac{1}{1-F(s)}$, i.e.
$w^*_s=\frac{l}{1-F(s)}$ and there exists a constant $l^*$ such that
$
\big(\int_0^{\theta(u^*)} \frac{w^*_s}{1-F(s)}dF(s)\big)^2=4 \varepsilon^2.
$
The existence of $l^*=2\varepsilon$
is verified directly.

Thus, the upper bound is valid:
$$
\varlimsup_{n\to\infty}\frac{1}{n^{1-2\alpha}}\log\mathsf{P}\big(\sup_{t\in[0,T]}n^\alpha|X^{n,\alpha}_t|
\ge \varepsilon\big) \le - 2\varepsilon^2.
$$

In order to complete the proof, we have to prove the following lower bound
\begin{gather*}
\varliminf_{n\to\infty}\frac{1}{n^{1-2\alpha}}\log\mathsf{P}\Big(\sup_{t\in[0,T]}n^\alpha|X^{n,\alpha}_t|
\ge \varepsilon\Big) \ge - 2\varepsilon^2
\end{gather*}
Formally, one may apply
\begin{gather*}
\varliminf_{n\to\infty}\frac{1}{n^{1-2\alpha}}\log\mathsf{P}\Big(\sup_{t\in[0,T]}n^\alpha|X^{n,\alpha}_t|
\ge
\varepsilon\Big)
\\
\ge -\frac{1}{2}\inf_{u\in \mathsf{C}^\circ}\int_0^{T\wedge
\theta(u)}\Big[\dot{u}_s+\frac{u_s}{1-F(s)} \Big]^2dF(s) ,
\end{gather*}
where $\mathsf{C}^\circ$ is an interior of $\mathsf{C}$. However,
$\mathsf{C}$ has an empty interior.
Fortunately, the proof of the upper bound gives us a hint: $F(\theta(u^*)=\frac{1}{2}$.
Choose $T^*$ with $F(T^*)=\frac{1}{2}$ and use an obvious inequality:
\begin{gather*}
\mathsf{P}\Big(\sup_{t\in[0,T]}n^\alpha|X^{n,\alpha}_t|
\ge
\varepsilon\Big) \ge \mathsf{P}\Big(n^\alpha|X^{n,\alpha}_{T^*}|
\ge
\varepsilon\Big).
\end{gather*}
Hence, only a lower bound
$
\varliminf\limits_{n\to\infty}\frac{1}{n^{1-2\alpha}}\log \mathsf{P}\big(n^\alpha|X^{n,\alpha}_{T^*}|
\ge
\varepsilon\big)\ge -2\varepsilon^2
$
has to be proven.
The latter is verified with the help of the large deviations principle for the different family
$
\{X^{n,\alpha}_{T^*}\}_{n\to\infty}.
$

Since
$
X^{n,\alpha}_{T^*}=\frac{1}{n^{1/2}-\alpha}\frac{1}{\sqrt{n}}\sum\limits_
{k=1}^n\big[I_{\{\xi_k\le T^*\}}-F(T^*)\big]
$
with
$
(I_{\{\xi_k\le T^*\}}-F(T^*))_{k\ge 1}
$
being the i.i.d. sequence of zero mean random variables having the variance
$
F(T^*)[1-F(T^*)]=\frac{1}{4},
$
the large deviations principle
for this family is well known and has the rate $\frac{1}{n^{1-2\alpha}}$ and the rate function
$
I(v)=\frac{v^2}{2F(T^*)[1-F(T^*)]}=2 v^2.
$
Therefore,
$$
\varliminf_{n\to\infty}\frac{1}{n^{1-2\alpha}}\log \mathsf{P}\Big(n^\alpha|X^{n,\alpha}_{T^*}|
\ge
\varepsilon\Big)= -\inf_{v:|v|\ge \varepsilon}I(v)=-2\varepsilon^2.
$$
\end{proof}

\appendix

\section{\bf Large deviations principle for $\pmb{X^{n,\alpha}}$  }
\label{sec-3}

By \eqref{eq:2.2a},
$
X^{n,\alpha}_t=-\int_0^t\frac{X^{n,\alpha}_s}{1-F(s)}dF(s)+
\frac{1}{n^{\frac{1}{2}-\alpha}}\mathbf{M}^n_t.
$
A complicated structure of the martingale $(\mathbf{M}^n_t,\mathscr{F}_t)_{t\in [0,1]}$
does not allow us to apply Freidlin and Wentzell's (1984), \cite{FW},  or of Wentzell's (1986) \cite{W}
results.

On the other hand, by Theorem \ref{lem-2.-1}, the family
$\{(\mathbf{M}^n_t)_{t\in[0,1]}\}_{n\to\infty}$ converges in law to Gaussian
martingale $(\mathbf{M}_t)_{t\in[0,1]}$ with $\langle
\mathbf{M}\rangle_t=F(t)$. Notice also that the family
$
\{(\frac{1}{n^{\frac{1}{2}-\alpha}}\mathbf{M}_t)_{t\in[0,1]}\}_{n\to \infty}
$
is in a framework of Freidlin and Wentzell (1984). So, it obeys the large deviations principle with the rate
$
\frac{1}{n^{1-2\alpha}}
$
and the rate function

\begin{equation}\label{V}
I(u)=
  \frac{1}{2}\begin{cases}
    \int_0^T\dot{u}^2_tdF(t) , &{\substack{u_0=0\\du_t=\dot{u}_tdF(t)\\
                                                                      \int_0^T\dot{u}^2_tdF(t)<\infty}}
 \\
    \infty , & \text{otherwise}.
  \end{cases}
\end{equation}

\begin{theorem}\label{theo-A.1}
For any $\alpha\in\big(0,\frac{1}{2}\big)$ and any $T$ in a small vicinity of $\{1\}$, the
families
$$
\Big\{\Big(\frac{1}{n^{\frac{1}{2}-\alpha}}\mathbf{M}_t\Big)_{t\in[0,T]}\Big\}_{n\to\infty}
\quad\text{and}\quad
\Big\{\Big(\frac{1}{n^{\frac{1}{2}-\alpha}}\mathbf{M}^n_t\Big)_{t\in[0,T]}\Big\}_{n\to\infty}
$$
share the same large deviations principle.
\end{theorem}
\begin{proof}
Hereafter, $\lambda(s)$ is a bounded measurable
function.

Since
$
\frac{1}{n^{\frac{1}{2}-\alpha}}\int_0^t\lambda(s)d\mathbf{M}_s
$
is a continuous Gaussian martingale with the predictable variation process
$
\frac{1}{n^{1-2\alpha}}\int_0^t\lambda^2(s)dF(s)=:2\widetilde{\mathscr{E}}^{n,\alpha}_t(\lambda),
$
the function
$
\widetilde{\mathscr{E}}^{n,\alpha}_t(\lambda)
$
is the Laplace transform of
$
\frac{1}{n^{\frac{1}{2}-\alpha}}\int_0^t\lambda(s)d\mathbf{M}_s.
$
Moreover,
a random process
$
\widetilde{\mathfrak{z}}^{n,\alpha}_t=\exp\Big(\frac{1}{n^{\frac{1}{2}-\alpha}}\int_0^t\lambda(s)
d\mathbf{M}_s-\log \widetilde{\mathscr{E}}^{n,\alpha}_t(\lambda)\Big)
$
is a martingale. In the case of
$
\frac{1}{n^{\frac{1}{2}-\alpha}}\int_0^t\lambda(s)d\mathbf{M}^n_s,
$
an explicit formula for the Laplace transform is unknown. However, a random process
$
\mathscr{E}^{n,\alpha}_t(\lambda)=\exp\Big(\int_0^t\big[e^{\frac{\lambda(s)}{n^{1-\alpha}}}-1-
\frac{\lambda(s)}{n^{1-\alpha}}\big]d\mathbf{A}^n_s\Big)
$
``exponentially compensates''
$
\frac{1}{n^{\frac{1}{2}-\alpha}}\int_0^t\lambda(s)d\mathbf{M}^n_s
$
up to a martingale in a sense that a random process
$
\mathfrak{z}^{n,\alpha}_t=\exp\Big(\frac{1}{n^{\frac{1}{2}-\alpha}}\int_0^t\lambda(s)
d\mathbf{M}^n_s
-\mathscr{E}^{n,\alpha}_t(\lambda)\Big)
$
is a local martingale (the latter is verified by applying the It\^o formula).

By a terminology of Puhalskii (1994, 2001), $\widetilde{\mathscr{E}}^{n,\alpha}_t(\lambda)$
and $\mathscr{E}^{n,\alpha}_t(\lambda)$ are referred to as ``Stochastic Exponentials'' related to
the families
$$
\Big\{\Big(\frac{1}{n^{\frac{1}{2}-\alpha}}\mathbf{M}_t\Big)_{t\in[0,T]}
\Big\}_{n\to\infty}\quad\text{and}\quad
\Big\{\Big(\frac{1}{n^{\frac{1}{2}-\alpha}}\mathbf{M}^n_t\Big)_{t\in[0,T]}\Big\}_{n\to\infty}
$$
respectively.

A role of stochastic exponential is revealed in
Puhalskii (1994, 2001). In our setting the Puhalskii result states that
the above-mentioned families share the same large deviations principle
provided that for any $\eta>0$
and any bounded $\lambda(t)$,
\begin{equation}\label{ojoj}
\lim_{n\to
0}\frac{1}{n^{1-2\alpha}}\log\mathsf{P}\bigg(\sup_{t\in[0,T]}
n^{1-2\alpha}\Big|\log\frac{\mathscr{E}^{n,\alpha}_t(\lambda)}{
\widetilde{\mathscr{E}}^{n,\alpha}_t(\lambda)}\Big|>\eta\bigg)=-\infty.
\end{equation}

We finish the proof by verification of \eqref{ojoj}.
Taking into account
$$
d\mathbf{A}^n_s=n\frac{1-F_n(s)}{1-F(s)}dF(s)
$$
(see \eqref{eq:AM}), write
\begin{gather*}
n^{1-2\alpha}\Big|\log\frac{\mathscr{E}^{n,\alpha}_t(\lambda)}{
\widetilde{\mathscr{E}}^{n,\alpha}_t(\lambda)}\Big|
\\
=n^{1-2\alpha}\Big|\int_0^t\Big[e^{\frac{\lambda(s)}{n^{1-\alpha}}}-1-
\frac{\lambda(s)}{n^{1-\alpha}}\Big]
d\mathbf{A}^n_s-\int_0^t\frac{\lambda^2(s)}{2n^{1-2\alpha}}dF(s)\Big|
\\
=
n^{1-2\alpha}\Big|\int_0^t\Big[e^{\frac{\lambda(s)}{n^{1-\alpha}}}-1-
\frac{\lambda(s)}{n^{1-\alpha}}-\frac{\lambda^2(s)}{2n^{2(1-\alpha)}}\Big]
n\frac{1-F_n(s)}{1-F(s)}dF(s)
\\
-\int_0^t\Big[\frac{\lambda^2(s)}{2n^{2(1-\alpha)}}n\frac{1-F_n(s)}{1-F(s)}
-\frac{\lambda^2(s)}{2n^{1-2\alpha}}\Big]dF(s)\Big|.
\end{gather*}
Since
\begin{gather*}
\int_0^Tn^{1-2\alpha}\Big|e^{\frac{\lambda(s)}{n^{1-\alpha}}}-1-
\frac{\lambda(s)}{n^{1-\alpha}}-\frac{\lambda^2(s)}{2n^{2(1-\alpha)}}\Big|
n\frac{1-F_n(s)}{1-F(s)}dF(s)
\\
\le \text{const.}\frac{1}{n^{1-\alpha}}\int_0^T\frac{1}{1-F(s)}dF(s)
=\text{const.}\frac{1}{n^{1-\alpha}}\log\frac{1}{1-F(T)}
\end{gather*}
and
\begin{align*}
&
\int_0^Tn^{1-2\alpha}\Big|\frac{\lambda^2(s)}{2n^{2(1-\alpha)}}n\frac{1-F_n(s)}{1-F(s)}
-\frac{\lambda^2(s)}{2n^{1-2\alpha}}\Big|dF(s)
\\
&\le \int_0^T\frac{\lambda^2(s)}{2}\frac{|F_n(s)-F(s)|}{1-F(s)}dF(s)
\\
&\le\text{const.}\sup_{s\in[0,T]}|F_n(s)-F(s)|\int_0^T\frac{1}{1-F(s)}dF(s)
\\
&=\text{const.}\sup_{s\in[0,T]}|F_n(s)-F(s)|\log\frac{1}{1-F(T)},
\end{align*}
we shall analyze an upper bound of the following inequality:
\begin{multline*}
n^{1-2\alpha}\Big|\sup_{t\in[0,T]}\log\frac{\mathscr{E}^{n,\alpha}_t(\lambda)}{
\widetilde{\mathscr{E}}^{n,\alpha}_t(\lambda)}\Big|
\\
\le
\text{const.}\log\frac{1}{1-F(T)}\Big[\frac{1}{n^{1-\alpha}}+\sup_{s\in[0,T]}|F_n(s)-F(s)|\Big].
\end{multline*}
Obviously, \eqref{ojoj} is valid if
$$
\lim\limits_{n\to\infty}\frac{1}{n^{1-2\alpha}}\mathsf{P}\big(\sup_{s\in[0,T]}|F_n(s)
-F(s)|>\eta-\frac{1}{n^{1-\alpha}}\big)=-\infty.
$$

For fixed $\eta$, let us choose a number $n_0$ such that
$\frac{1}{n_0^{1-\alpha}}\le \frac{\eta}{2}$ and all $n\ge n_0$. In this scenario it remains to show that
\begin{equation*}
\lim_{n\to\infty}\frac{1}{n^{1-2\alpha}}\mathsf{P}\Big(\sup_{s\in[0,T]}|F_n(s)
-F(s)|>\frac{\eta}{2}\Big)=-\infty.
\end{equation*}
The latter heavily uses Kolmogorov's bound:
\begin{multline*}\
\frac{1}{n^{1-2\alpha}}\log\mathsf{P}\bigg(\sup_{s\in[0,T]}|F_n(s)-F(s)|\ge
\frac{\eta}{2}\bigg)\\
\le
\frac{\log 2}{n^{1-2\alpha}}
-n^{2\alpha}\Big[\frac{\eta}{16}\Big\{\log\Big(1+\frac{\eta^2}{128}\Big)-1\Big\}
+\frac{8}{\eta}\log\Big(1+\frac{\eta^2}{128}\Big)\Big]\bigg)
\xrightarrow[n\to\infty]{}-\infty.
\end{multline*}
\end{proof}

Theorem \ref{theo-A.1} implies the following result.

\begin{theorem}\label{theo-4.2x}
For any $\alpha\in\big(0,\frac{1}{2}\big)$ and any $T$ in a small vicinity of $\{1\}$,
the
family $\{(X^{n,\alpha}_t)_{t\in[0,T]}\}_{n\to\infty}$ obeys the large deviations principle
in the Skorokhod space $\mathbb{D}_{[0,T]}$ relative Skorokhod's
and uniform metrics with the rate speed
$\frac{1}{n^{1-2\alpha}}$ and the rate function
\begin{equation*}
J_T(u)=
  \frac{1}{2}\begin{cases}
    \displaystyle{\int_0^T}\Big(\dot{u}_t+\frac{u_t}{1-F(t)}\Big)^2dt , &{\substack{u_0=0\\du_t=
                                                                      \dot{u}_tdF(t)\\
                                                                      \int_0^T(\dot{u}_t+\frac{u_t}{1-F(t)}
                                                                      )^2dF(t)<\infty}}
 \\
    \infty , & \text{otherwise}.
  \end{cases}
\end{equation*}
\end{theorem}
\begin{proof}
By Lemma \ref{lem-2.1},
$
X^{n,\alpha}_t=\mathsf{\Psi}\Big(\frac{1}{n^{\frac{1}{2}-\alpha}}\mathbf{M}^{n}_{[0,t]}\Big).
$
Hence and by Theorem \ref{theo-A.1} the
family $\{(X^{n,\alpha}_t)_{t\in[0,T]}\}_{n\to\infty}$ shares the large deviations principle with
the
family $\mathsf{\Psi}\Big(\frac{1}{n^{\frac{1}{2}-\alpha}}\mathbf{M}_{[0,t]}\Big)_{t\in[0,T]}$.

Hence, by the contraction principle of Varadhan (1984) and \eqref{V}
$
J_T(u)=I(v)_{v=\mathsf{\Psi}(u)}$.
\end{proof}

\end{document}